\documentclass[12pt]{article}
\usepackage{amsmath,amscd,amsthm,a4,amssymb}

\newtheorem{theorem}{Theorem}{}
\newtheorem{lemma}{Lemma}{}
\newtheorem{corollary}{Corollary}{}
\newtheorem{definition}{Definition}{}
{}
{}

\newcommand{\es}{\mathop{\rm ess \; sup}\limits}
\newcommand{\ei}{\mathop{\rm ess \; inf}\limits}

\def\Rn{{\mathbb{R}^n}}
\def\a{\alpha}

\usepackage{color}
\usepackage{srcltx}
\usepackage{times}

\newcommand{\al}{\alpha}

\newcommand{\om}{\omega}
\newcommand{\Om}{\Omega}
\newcommand{\lb}{\lambda}

\newcommand{\vi}{\varphi}

\newcommand{\vs}{\vspace}

\usepackage{color}

\begin{document}

\begin{center}
\Large Commutators of Riesz potential in the vanishing generalized weighted Morrey spaces with variable exponent
\end{center}

\

\centerline{\large Vagif S. Guliyev$^{a,b,c,}$\footnote{
Corresponding author.
\\
E-mail adresses: vagif@guliyev.com (V.S. Guliyev), hasanovjavanshir@yahoo.com.tr (J.J. Hasanov), xayyambadalov@gmail.com (X.A. Badalov).}, Javanshir J. Hasanov$^{d}$, Xayyam A. Badalov$^{c}$ }

\

\centerline{$^{a}$\it Dumlupinar University, Department of Mathematics, 43020 Kutahya, Turkey}

\centerline{$^{b}$\it \small S.M. Nikolskii Institute of Mathematics at RUDN University, 6 Miklukho-Maklay St, Moscow, 117198}

\centerline{$^{c}$\it Institute of Mathematics and Mechanics of NAS of Azerbaijan, AZ1141 Baku, Azerbaijan}

\centerline{$^{d}$\it Azerbaijan State Oil and Industry University, Azadlig av.20, AZ 1601, Baku, Azerbaijan}

\

%
%
%
%
%

\begin{abstract}
Let $\Om \subset \Rn$  be an unbounded open set. We consider the generalized weighted Morrey spaces $\mathcal{M}^{p(\cdot),\vi}_{\om}(\Om)$ and the vanishing generalized weighted Morrey spaces $V\mathcal{M}^{p(\cdot),\vi}_{\om}(\Om)$
with variable exponent $p(x)$ and a general function $\vi(x,r)$ defining the Morrey-type norm.
The main result of this paper are the boundedness of Riesz potential and its commutators on the spaces $\mathcal{M}^{p(\cdot),\vi}_{\om}(\Om)$ and $V\mathcal{M}^{p(\cdot),\vi}_{\om}(\Om)$.
This result generalizes several existing results for Riesz potential and its commutators on Morrey type spaces. Especially, it gives a unified result for
generalized Morrey spaces and variable Morrey spaces which currently gained a
lot of attentions from researchers in theory of function spaces.
\end{abstract}

\

\noindent{\bf AMS Mathematics Subject Classification:} $~~$ 42B20, 42B25, 42B35

\noindent{\bf Key words:} {Riesz potential, commutator, vanishing generalized weighted Morrey space with variable exponent, $BMO$ space}

\section{Introduction}

\

The variable exponent generalized weighted Morrey spaces $\mathcal{M}^{p(\cdot),\vi}_{\om}(\Om)$ over an open set $\Omega \subset \Rn$ was introduced and
the boundedness of the Hardy-Littlewood maximal operator, the singular integral operators and their commutators on these spaces was proven in \cite{GulHasBad}.
The main focus of this article is to prove that the Riesz potential and its commutators are bounded on generalized weighted Morrey spaces $\mathcal{M}^{p(\cdot),\vi}_{\om}(\Om)$ and vanishing generalized weighted Morrey spaces $V\mathcal{M}^{p(\cdot),\vi}_{\om}(\Om)$ with variable exponents. Also some Sobolev-type inequalities for Riesz potentials on spaces $\mathcal{M}^{p(\cdot),\vi}_{\om}(\Om)$ and $V\mathcal{M}^{p(\cdot),\vi}_{\om}(\Om)$ are proved.

The classical Morrey spaces were introduced by Morrey \cite{Morrey} to study the local behavior of solutions to second-order
elliptic partial differential equations. Moreover, various Morrey-type spaces are
defined in the process of study. Mizuhara \cite{Miz} and Nakai \cite{Nakai}
introduced generalized Morrey spaces $M^{p,\varphi}(\Rn)$ (see, also \cite{GulJIA}); Komori and Shirai \cite{KomShir} defined weighted Morrey spaces $L^{p,\kappa}(w)$; Guliyev \cite{Gul_EMJ2012} gave a concept of the generalized weighted Morrey spaces $M^{p,\varphi}_{w}({\mathbb R}^n)$ which could be viewed as extension of both $M^{p,\varphi}(\Rn)$ and $L^{p,\kappa}(w)$.

Vanishing Morrey spaces $VM^{p,\varphi}(\Rn)$ are subspaces of functions in Morrey spaces which were introduced by Vitanza \cite{Vi} satisfying the condition
$$
\lim\limits_{r\to 0}\sup\limits_{\substack{x\in\Rn \\ 0<t<r}} r^{-\frac{\lambda}{p}} \|f\chi_{B(x,t)}\|_{L^{p(\cdot)}(B(x,t)}=0
$$
and applied there to obtain a regularity result for elliptic partial differential
equations. Also Ragusa \cite{Ra} proved a sufficient condition for commutators of fractional integral operators to belong to vanishing Morrey spaces $VM^{p,\lambda}(\Rn)$. The vanishing generalized Morrey spaces $VM^{p,\varphi}(\Rn)$ were introduced and studied by Samko in \cite{NS}, see also \cite{AkbGulOm2017, DerGulSam2015, LongHan}.

As it is known, last two decades there is an increasing interest to the study of
variable exponent spaces and operators with variable parameters in such
spaces, we refer for instance to the surveying papers \cite{Din3, KSAMADE, Sam4},  on the progress in this field, including topics
of Harmonic Analysis and Operator Theory, see also references therein.
For mapping properties of maximal functions and Riesz potential on Lebesgue
spaces with variable exponent we refer to \cite{CFMP, CF2013, Din, 160zc, Din2, KSAM, Kop, RaTach}.

Variable exponent Morrey spaces ${\cal  L}^{p(\cdot), \lambda(\cdot)}(\Om)$, were introduced and studied in \cite{AHS} and \cite{MizShim} in the Euclidean setting. The boundedness of Riesz potential in variable exponent Morrey spaces ${\cal L}^{p(\cdot),\lb(\cdot)}(\Om)$ under the log-condition on $p(\cdot),\lb(\cdot)$  and a Sobolev type ${\cal L}^{p(\cdot),\lambda(\cdot)}\rightarrow {\cal L}^{q(\cdot),\lambda(\cdot)}$--theorem for potential operators of variable order $\al(x)$ was proved in \cite{KSAM}. In the case of constant $\al$, there
was also proved a boundedness theorem in the limiting case $p(x) =
\frac{n-\lambda(x)}{\al},$ when the potential operator $I^{\al}$ acts from
${\cal L}^{p(\cdot),\lambda(\cdot)}$ into $BMO$  was proved in \cite{AHS}. In \cite{MizShim} the
maximal operator and potential operators were considered in a somewhat more
general space, but under more restrictive conditions on $p(x)$.

Generalized Morrey spaces of such a kind in the case of constant $p$ were
studied  in \cite{Er2013BVP, GulJIA, Miz, Nakai, Scap1, Scap2}. In the case of bounded sets  the boundedness of the maximal operator, singular integral operator and the potential operators in generalized variable exponent Morrey type spaces was proved in  \cite{GulHasSam, GulHasSam1, GulHasS2} and in the case of unbounded sets  in \cite{GulS}.
Also, in the case of bounded sets the boundedness of these operators in generalized variable exponent weighted Morrey spaces
for the power weights was proved in \cite{Has7}.

In the case of constant $p$ and $\lb$, the results on the boundedness of potential operators go back to \cite{Adams} and \cite{P}, respectively, while the
boundedness of the maximal operator in the Euclidean setting was proved in
\cite{ChiFra}; for further results  in the case of constant $p$  and $\lb$ (see, for instance, \cite{BurGul2, ErGulAs2017}).

In the spaces $\mathcal{M}^{p(\cdot),\vi}_{\om}(\Om)$ over open sets $\Om \subset \Rn$ we
consider the following operators: \\
1) the Hardy-Littlewood maximal operator
$$
Mf(x)=\sup\limits_{r>0}|B(x,r)|^{-1} \int_{\widetilde{B}(x,r)}|f(y)|dy,
$$
2) Riesz potential operator
$$
I^{\a} f(x)=\int_{\Om}|x-y|^{\a -n} f(y)dy, ~~ 0<\a <n,
$$
3) the fractional maximal operator
$$
M^{\a}f(x)=\sup\limits_{r>0}|B(x,r)|^{\frac{\a}{n}-1}
\int_{\widetilde{B}(x,r)}|f(y)|dy, \,\, 0\le\a <n.
$$

We find the condition on the function $\vi(x,r)$ for the boundedness of the Riesz potential $I^{\al}$ and its commutators in the generalized weighted Morrey space ${ \cal M}^{p(\cdot),\vi}_{\om}(\Om)$
and the vanishing generalized weighted Morrey spaces $V\mathcal{M}^{p(\cdot),\vi}_{\om}(\Om)$ with variable $p(x)$ under the log-condition on $p(\cdot)$.

\vs{1mm}The paper is organized as follows. In Section \ref{Preliminaries} we
provide necessary preliminaries on variable exponent weighted Lebesgue, generalized weighted Morrey spaces
and vanishing generalized weighted Morrey spaces. In Section \ref{potentials} we prove the boundedness of
Riesz potential and its commutators on the variable exponent generalized weighted Morrey spaces.
In Section \ref{potentialsv} we prove the boundedness of Riesz potential and its commutators on the variable exponent vanishing generalized weighted Morrey spaces.

The main results are given in Theorems \ref{potv}, \ref{M1Xv}, \ref{RKomv}, \ref{potvk}, \ref{M1Xvc}, \ref{Risv} and \ref{CRisv}. We
emphasize that the results we obtain for generalized Morrey spaces are new
even in the case when $p(x)$ is constant, because we do not impose any
monotonicity type condition on $\vi(r).$ 

We use the following notation: $\Rn$ is the $n$-dimensional Euclidean space,  $\Om \subset \Rn$
is an open set, $\chi_E(x)$ is the characteristic function of a set $E\subseteq
\Rn$,  $B(x,r)=\{y \in \Rn :|x-y| < r\}), \  \widetilde{B}(x,r)=B(x,r)\cap\Om$,
by $c$,$C, c_1,c_2$ etc, we denote various absolute  positive constants,
which may have different values even in the same line.

\

\section{Preliminaries on variable exponent weighted Lebesgue, generalized weighted Morrey spaces and vanishing generalized weighted Morrey spaces}\label{Preliminaries}

\setcounter{theorem}{0} \setcounter{equation}{0}

\

We refer to the book \cite{160zc} for variable exponent Lebesgue spaces but give some basic definitions and facts. Let $p(\cdot)$ be a measurable function on $\Om$
with values in $(1,\infty)$. {An open set $\Om$ which may be unbounded throughout the whole paper}. We mainly suppose that
\begin{equation}\label{h0}
1< p_-\le p(x)\le p_+<\infty,
\end{equation}
where $
p_-:=\ei_{x\in \Om}p(x)$, $p_+:=\es_{x\in\Om}p(x)$.
By $L^{p(\cdot)}(\Om)$ we denote the space of all measurable functions $f(x)$
on $\Om$ such that
$$
I_{p(\cdot)}(f)= \int_{\Om}|f(x)|^{p(x)} dx < \infty.
$$
Equipped with the norm
$$\|f\|_{p(\cdot)}=\inf\left\{\eta>0:~I_{p(\cdot)}\left(\frac{f}{\eta}\right)\le
1\right\},
$$
this is a Banach function space. By $p^\prime(\cdot) =\frac{p(x)}{p(x)-1}$, $x\in
\Om,$ we denote the conjugate exponent.

The space $L^{p(\cdot)}(\Om)$ coincides with the space
\begin{equation}\label{hs1}
\left\{f(x): \left|\int_\Omega f(y)g(y)dy\right|<\infty~~for~
all~~g\in L^{p'(\cdot)}(\Omega)\right\}
\end{equation}
up to the equivalence of the norms
\begin{equation}\label{hs2}
\|f\|_{L^{p(\cdot)}(\Om)} \approx \sup_{\|g\|_{L^{p'(\cdot)}}\le 1}
\left|\int_{\Omega}f(y)g(y)dy\right|,
\end{equation}
see \cite[Proposition 2.2]{Kwok2016},  see also
\cite[Theorem 2.3]{KovRak1},  or \cite[Theorem 3.5]{Sam6}.

For the basics on variable exponent Lebesgue spaces we refer to \cite{Sh}, \cite{KovRak1}.

\noindent ${\cal P}(\Om)$ is the set of bounded measurable functions $p :\Om\to  [1,\infty)$;

\noindent ${\cal P}^{log}(\Om)$ is the set of exponents $p\in {\cal P}(\Om)$ satisfying the local log-condition
\begin{equation}\label{h8}
|p(x)-p(y)|\le\frac{A}{-\ln |x-y|}, \;\; |x-y|\le \frac{1}{2} \;\; x,y\in \Om,
\end{equation}
where  $A=A(p) > 0$ does not depend on $x, y$;

\noindent${\cal A}^{log}(\Om)$ is the set of bounded exponents $p :\Om\to  \Rn$ satisfying the
condition \eqref{h8};

\noindent$\mathbb{P}^{log}(\Om)$ is the set of exponents $p\in {\cal P}^{log}(\Om)$ with $1< p_-\le  p_+<\infty$;

\noindent for $\Om$ which may be unbounded, by ${\cal P}_\infty(\Om)$, ${\cal P}^{log}_\infty(\Om)$, $\mathbb{P}^{log}_\infty(\Om)$, ${\cal A}^{log}_\infty(\Om)$
we denote the subsets of the above sets of exponents satisfying the
decay condition (when $\Om$ is unbounded)
\begin{equation}\label{hi8}
|p(x)-p(\infty)|\le\frac{A_\infty}{\ln (2+|x|)},  \;\; x\in \Rn,
\end{equation}
where $p_\infty=\lim\limits_{x\to \infty}p(x)>1$.

We will also make use of the estimate provided by the following lemma ( see
\cite{160zc}, Corollary 4.5.9).
\begin{equation}\label{estikmate1}
\|\chi_{\widetilde{B}(x,r)}(\cdot)\|_{p(\cdot)}\le C
r^{\theta_p(x,r)}, \quad \ x\in\Om, \,\,p\in \mathbb{P}^{log}_\infty(\Om),
\end{equation}
where $\theta_p(x,r)=\left\{\begin{array}{cc} \frac{n}{p(x)}, \,\, r\le 1, \\  \frac{n}{p(\infty)}, \,\, r> 1\end{array}\right.$.

By $\om$ we always denote a weight, i.e. a {positive}, locally
integrable function with {domain} $\Om$.
The weighted Lebesgue space $L^{p(\cdot)}_{\om}(\Om)$ is defined as the set of all measurable
functions for which
\begin{equation*}
\|f\|_{L^{p(\cdot)}_{\om}(\Om)}= \|f\om\|_{L^{p(\cdot)}(\Om)}.
\end{equation*}

Let us define the class $A_{p(\cdot)}(\Om)$ (see \cite{Din3}, \cite{Kop}) to consist of those weights $\om$ for which
$$
[\om]_{A_{p(\cdot)}} \equiv \sup_{B}|B|^{-1}\|\om\|_{L^{p(\cdot)}(\widetilde{B}(x,r))}
\|\om^{-1}\|_{L^{p'(\cdot)}(\widetilde{B}(x,r))}<\infty.
$$

A weight function $\om$ belongs to the class $A_{p(\cdot),q(\cdot)}(\Om)$ if
$$
[\om]_{A_{p(\cdot),q(\cdot)}} \equiv \sup_{B}|B|^{\frac{1}{p(x)}-\frac{1}{q(x)}-1}\|\om\|_{L^{q(\cdot)}(\widetilde{B}(x,r))}
\|\om^{-1}\|_{L^{p'(\cdot)}(\widetilde{B}(x,r))}<\infty.
$$
\begin{lemma} Let $p,q$  satisfy condition \eqref{h0} and $\om \in A_{p(\cdot),q(\cdot)}(\Om)$, then $\om^{-1}\in A_{q'(\cdot),p'(\cdot)}(\Om)$, with $\frac1{p(x)}+\frac 1{p'(x)}=1$.
\end{lemma}
\begin{proof}
Let $p,q$  satisfy condition \eqref{h0} and $\om \in A_{p(\cdot),q(\cdot)}(\Om)$.
Then $\vi=\om^{-1} \in A_{q'(\cdot),p'(\cdot)}(\Om)$. Indeed,
\begin{align*}	
&|B|^{\frac{1}{q'(x)}-\frac{1}{p'(x)}-1} \, \|\vi\|_{L^{p'(\cdot)}(\widetilde{B}(x,r))} \, \|\vi^{-1}\|_{L^{q(\cdot)}(\widetilde{B}(x,r))}
\\
= &	|B|^{\frac{1}{p(x)}-\frac{1}{q(x)}-1}
\|\om\|_{L^{q(\cdot)}(\widetilde{B}(x,r))}
\|\om^{-1}\|_{L^{p'(\cdot)}(\widetilde{B}(x,r))}.
\end{align*}
\end{proof}

\begin{theorem}\label{max2} \cite[Therem 1.1]{HD} Let $\Om \subset \Rn$ be an open unbounded  set  and  $p\in \mathbb{P}^{log}_\infty(\Om)$. Then $M:L^{p(\cdot)}_{\om}(\Om)\to L^{p(\cdot)}_{\om}(\Om)$ if and only if $\om\in A_{p(\cdot)}(\Om)$.
\end{theorem}

For unbounded sets, say $\Om = \Rn,$ and constant orders a the corresponding Sobolev theorem
proved in \cite{CaCF,CFMP} runs as follows.
\begin{theorem}\label{KS1} Let $\Om \subset \Rn$ be an open unbounded set, $0<\al<n$ and $p\in \mathbb{P}^{log}_\infty(\Om)$. Let also $p_+<\frac {n}{\al} $.
	Then the operators $M^{\al}$ and $I^{\al}$ are bounded from $L^{p(\cdot)}(\Om)$ to
	$L^{q(\cdot)}(\Om)$ with $\frac1{q(x)}=\frac 1{p(x)}-\frac {\al} {n}.$
\end{theorem}

\vs{4mm} Let $\lambda(x)$ be a measurable function on $\Omega$ with values in $[0,n]$. The variable Morrey space $\mathcal{L}^{p(\cdot),\lambda(%
	\cdot)}(\Omega)$ and variable weighted Morrey space $\mathcal{L}%
^{p(\cdot),\lambda(\cdot)}_{\om}(\Omega)$ is defined as the set of
integrable functions $f$ on $\Omega$ with the finite norms
\begin{equation*}
\|f\|_{\mathcal{L}^{p(\cdot),\lambda(\cdot)}(\Omega)}= \sup_{x\in \Omega, \;
	t>0} t^{-\frac{\lambda(x)}{p(x)}}\| f\chi_{\widetilde{B}(x,t)}\|_{L^{p(%
		\cdot)}(\Omega)},
\end{equation*}
\begin{equation*}
\|f\|_{\mathcal{L}^{p(\cdot),\lambda(\cdot)}_{\om}(\Omega)}=
\sup_{x\in \Omega, \; t>0} t^{-\frac{\lambda(x)}{p(x)}}\| f\chi_{\widetilde{B}(x,t)}\|_{L^{p(\cdot)}_{\om}(\Omega)}.
\end{equation*}

Let $\om$ be a nonnegative measurable function on $\Rn$ such that $\om^p$ is locally integrable on $\Rn$. Then a Radon
measure $\mu$ is canonically associated with the weight $\om(\cdot)^{p(\cdot)}$, that is,
$$
\mu(E) =\int_E \om(y)^{p( y)}dy.
$$

We denote by ${\cal L}^{p(\cdot),\lambda}(\Rn,d\mu)$ the set of all measurable functions $f$ with finite norm
$$
\|f\|_{{\cal L}^{p(\cdot),\lambda}(\Rn,d\mu)}= \inf\left\{\eta>0:~\sup_{x\in \Om, \; t>0}
\frac{t^{\lambda}}{\mu(B(x,t))}\int_{B(x,t)}
\left(\frac{|f(y)|}{\eta}\right)^{p(y)}d\mu(y) \le 1\right\}.
$$

\begin{theorem}\label{pv} \cite{MizSh}
	Let $0<\al<n$, $0\le \lb<n$, $p\in \mathbb{P}^{log}_\infty(\Rn)$, $p_+<\frac {\lb}{\al} $, $\frac1{q(x)}=\frac 1{p(x)}-\frac {\al} {\lb}$, $\om\in A_{p(\cdot)}(\Rn)$.
	Then the operator  $I^{\al}$ is bounded from ${\cal L}^{p(\cdot),\lb} (\Rn,d\mu)$ to
	${\cal L}^{q(\cdot),\lb} (\Rn,d\mu)$.
\end{theorem}

In view of the well known pointwise estimate $M^{\al}f(x)\le  C(I^{\al}|f|)(x)$, it suffices to treat only the case of the operator
$I^{\al}$.

\begin{corollary}\label{pvc} (\cite{CrWa}, \cite{MizSh})
	Let $\Om \subset \Rn$ be an open unbounded  set, $0<\al<n$, $p\in \mathbb{P}^{log}_\infty(\Om)$, $p_+<\frac {n}{\al} $, $\frac1{q(x)}=\frac 1{p(x)}-\frac {\al} {n}$, $\om\in A_{p(\cdot),q(\cdot)}(\Om)$.
	Then the operators $M^{\al}$ and $I^{\al}$ are bounded from $L^{p(\cdot)}_{\om}(\Om)$ to
	$L^{q(\cdot)}_{\om}(\Om)$.
\end{corollary}

Let $M^\sharp$ be the sharp maximal function defined by
\begin{equation*}
M^\sharp f(x)=\sup\limits_{r>0}|B(x,r)|^{-1} \int_{\widetilde{B}%
	(x,r)}|f(y)-f_{\widetilde{B}(x,r)}|dy,
\end{equation*}
where $f_{\widetilde{B}(x,t)}(x)=|\widetilde{B}(x,t)|^{-1}  \int_{\widetilde{B}(x,t)} f(y) dy$.
\begin{definition}
	We define the $BMO(\Om)$ space  as the set of all
	locally integrable functions $f$ with finite norm
	$$
	\|f\|_{BMO}=\sup\limits_{x\in \Om}
	M^\sharp f(x)= \sup\limits_{x\in \Omega, \;r>0}|B(x,r)|^{-1} \int_{\widetilde{B}(x,r)}|f(y)-f_{\widetilde{B}(x,r)}|dy.
	$$
\end{definition}

\begin{definition}
	We define the $BMO_{p(\cdot),\om}(\Om)$ space  as the set of all
	locally integrable functions $f$ with finite norm
	$$
	\|f\|_{BMO_{p(\cdot),\om}}=\sup\limits_{x\in \Omega, \;r>0}\frac{\|(f(\cdot)-f_{\widetilde{B}(x,r)})\chi_{\widetilde{B}(x,r)}\|_{L^{p(\cdot)}_{\om}(\Om)}}
	{\|\chi_{\widetilde{B}(x,r)}\|_{L^{p(\cdot)}_{\om}(\Om)}}.
	$$
\end{definition}

\begin{theorem}\label{bmow}  \cite{Kwok2016} Let $\Om \subset \Rn$ be an open unbounded  set,  $p\in \mathbb{P}^{log}_\infty(\Om)$ and $\om$ be
	a Lebesgue measurable function. If $\om\in A_{p(\cdot)}(\Om)$, then the norms $\|\cdot\|_{BMO_{p(\cdot),\om}}$  and $\|\cdot\|_{BMO}$
	are mutually equivalent.
\end{theorem}

Everywhere in the sequel the functions $\vi(x,r),\ \vi_1(x,r)$ and
$\vi_2(x,r)$ used in the body of the paper, are non-negative measurable function on $\Om \times (0,\infty)$.
We find it convenient to define the variable exponent generalized weighted Morrey spaces in the form as follows.
\begin{definition}
	Let  $1 \le p(x) < \infty$, $x\in \Om$.  The variable exponent generalized Morrey space
	$\mathcal{M}^{p(\cdot),\vi}(\Om)$ and variable exponent generalized weighted Morrey space $\mathcal{M}^{p(\cdot),\vi}_{\om}(\Om)$ are  defined by the norms
	$$
	\|f\|_{\mathcal{M}^{p(\cdot),\vi}} =
	\sup\limits_{x\in\Om,r>0}\frac{1}{\vi(x,r)r^{\theta_{p}(x,r)}}
	\|f\|_{L^{p(\cdot)}(\widetilde{B}(x,r))},
	$$
	$$
	\|f\|_{\mathcal{M}^{p(\cdot),\vi}_{\om}} =
	\sup\limits_{x\in\Om,r>0}\frac{1}{\vi(x,r)\|\om\|_{L^{p(\cdot)}(\widetilde{B}(x,r))}}
	\|f\|_{L^{p(\cdot)}_{\om}(\widetilde{B}(x,r))}.
	$$
\end{definition}
According to this definition, we recover the space
$\mathcal{L}^{p(\cdot),\lambda(\cdot)} (\Om)$ under the choice
$\vi(x,r)=r^{\theta_{p}(x,r)-\frac{\lb(x)}{p(x)}}$:
$$\mathcal{L}^{p(\cdot),\lambda(\cdot)} (\Om)=
\mathcal{M}^{p(\cdot),\vi(\cdot)}
(\Om)\Bigg|_{\vi(x,r)=r^{\theta_{p}(x,r)-\frac{\lb(x)}{p(x)}}}.$$

\begin{definition} (Vanishing generalized weighted Morrey space) The vanishing generalized weighted Morrey space $V\mathcal{M}^{p(\cdot),\vi}_{\om}(\Om)$ is defined as the space of functions $f\in \mathcal{M}^{p(\cdot),\vi}_{\om}(\Om)$ such
	that
	$$
	\lim\limits_{r\to 0}\sup\limits_{x\in\Om} \frac{1}{\vi_1(x,t)\|\om\|_{L^{p(\cdot)}(\widetilde{B}(x,t))}} \|f\chi_{\widetilde{B}(x,t)}\|_{L^{p(\cdot)}_{\om}(\Om)}=0.
	$$
\end{definition}

Everywhere in the sequel we assume that
\begin{equation}\label{van1}
\lim\limits_{r\to 0} \frac{1}{\|\om\|_{L^{p(\cdot)}(\widetilde{B}(x,t))}\inf\limits_{x\in\Om}\vi(x,t)} =0
\end{equation}
and
\begin{equation}\label{van1}
\sup\limits_{0<r<\infty} \frac{1}{\|\om\|_{L^{p(\cdot)}(\widetilde{B}(x,t))}\inf\limits_{x\in\Om}\vi(x,t)} =0,
\end{equation}
which makes the spaces $V\mathcal{M}^{p(\cdot),\vi}_{\om}(\Om)$  non-trivial, because bounded functions
with compact support belong then to this space.

\begin{theorem}\label{M1}  \cite{GulHasBad}
	Let $\Om \subset \Rn$ be an open unbounded  set,  $p\in \mathbb{P}^{log}_\infty(\Om)$, $\om\in A_{p(\cdot)}(\Om)$ and the function $\vi_1(x,r)$ and $\vi_2(x,r)$ satisfy the condition
	\begin{equation}\label{qhs1sh}
	\sup_{t>r}\frac{\ei_{t<s<\infty} \vi_1(x, s) \|\om\|_{L^{p(\cdot)}(\widetilde{B}(x,s))}}{\|\om\|_{L^{p(\cdot)}(\widetilde{B}(x,t))}}\le C\vi_2(x,r).
	\end{equation}
	
	Then the maximal operator $M$ is
	bounded from the space $\mathcal{M}^{p(\cdot),\vi_1}_{\om} (\Om)$ the space
	$\mathcal{M}^{p(\cdot),\vi_2}_{\om} (\Om)$.
\end{theorem}

\

\section{Riesz potential and its commutators in the spaces
	$\mathcal{M}^{p(\cdot),\vi}_{\om}(\Om)$}\label{potentials}
\setcounter{theorem}{0} \setcounter{equation}{0}

It is well-known that the commutator is an important integral operator and it plays a key role in harmonic analysis. Let $b \in BMO(\Rn)$. A well known result of Chanillo \cite{Chanillo} states that the commutator operator $[b,I^{\alpha}] f = I^{\alpha} (bf) - b \, I^{\alpha} f$ is bounded from $L_p(\Rn)$ to $L_q(\Rn)$ with $1/q=1/p-\alpha/n$, $1<p<n/\alpha$.

Let  $L^\infty(\mathbb{R}_+,v)$ be the weighted $L^\infty$-space  with the norm
$$
\|g\|_{L^\infty(\mathbb{R}_+,v)} = \es_{t>0}v(t)g(t).
$$

In the sequel $\mathfrak{M}(\mathbb{R}_+),$   $\mathfrak{M}^+(\mathbb{R}_+)$ and
$\mathfrak{M}^+\!(\mathbb{R}_+;\!\uparrow\!)\!$ stand for the set of
Lebesgue-measurable functions on  $\mathbb{R}_+$, and its subspaces of  nonnegative and nonnegative non-decreasing
functions, respectively. We also denote
$$
{\mathbb A}=\left\{\varphi \in {\mathfrak M}^+(\mathbb{R}_+;\uparrow):
\lim_{t\rightarrow 0+}\varphi(t)=0\right\}.
$$
Let $u$ be a continuous and non-negative function on $\mathbb{R}_+$. We
define the supremal operator $\overline{S}_{u}$  by
$$
(\overline{S}_{u}g)(t): = \|u\, g\|_{L_{\i}(0,t)},~~t\in (0,\infty).
$$
In the following theorem  proved in \cite{BurGogGulMust1}, we use the notation
$$
\widetilde{v}_1(t)=\sup\limits_{0<\xi<t}v_1(\xi).
$$

\begin{theorem}\label{thm5.1}
	Suppose that $v_1$ and $v_2$ are nonnegative measurable functions such that $0 <
	\|v_1\|_{L_\infty(0,t)} <\infty$ for every $t > 0$. Let $u$ be a continuous nonnegative function on $\mathbb{R}$. Then the
	operator $\overline{S}_{u}$ is bounded from $L_\infty(\mathbb{R}_+, v_1)$ to $L_\infty(\mathbb{R}_+, v_2)$ on the cone ${\mathbb A}$ if and only if
	\begin{equation*}\label{eq6.8}
	\left\|v_2 \overline{S}_{u}\left(  \| v_1 \|^{-1}_{L_{\infty}(0,\cdot)}
	\right)\right\|_{L_{\infty}(\mathbb{R}_+)}<\infty.
	\end{equation*}
\end{theorem}

We will  use the following statement on the boundedness
of the weighted Hardy operator
$$
H_{w} g(t):=\int_t^{\i} g(s) w(s) ds, \,\,\, H^{\ast}_{w} g(t):=\int_t^{\i} \Big(1+\ln\frac{s}{t}\Big) g(s) w(s) ds,\,\,  \ \  0<t<\infty,
$$
where $w$ is a weight.

The following theorem was proved in \cite{GulJMS2013}.

\begin{theorem} \label{GulJMS}  \cite{GulJMS2013}
	Let $v_1, v_2$ and $w$ be weights on  $(0,\infty)$ and $v_1(t)$ be bounded outside a neighborhood of the
	origin. The inequality
	\begin{equation*}
	\sup _{t>0} v_2(t) H_{w} g(t) \leq C \sup _{t>0} v_1(t) g(t)
	\end{equation*}
	holds for some $C>0$ for all non-negative and non-decreasing $g$ on $(0,\infty)$ if and
	only if
	\begin{equation*}
	B:= \sup _{t>0} v_2(t)\int_t^{\i} \frac{w(s) ds}{\sup _{s<\tau<\infty} v_1(\tau)}<\infty.
	\end{equation*}
\end{theorem}

\begin{theorem}\label{thm3.2.}  \cite{Gul_EMJ2012}
	Let $v_1$, $v_2$ and $w$ be weights on  $(0,\infty)$ and $v_1(t)$ be bounded outside a neighborhood of the
	origin. The inequality
	\begin{equation} \label{vav01}
	\sup _{t>0} v_2(t) H^{\ast}_{w} g(t) \leq C \sup _{t>0} v_1(t) g(t)
	\end{equation}
	holds for some $C>0$ for all non-negative and non-decreasing $g$ on $(0,\infty)$ if and
	only if
	\begin{equation*} \label{vav02}
	B:= \sup _{t>0} v_2(t)\int_t^{\i} \Big(1+\ln\frac{s}{t}\Big) \frac{w(s) ds}{\sup _{0<\tau<s} v_1(\tau)}<\infty.
	\end{equation*}
	Moreover, the value  $C=B$ is the best constant for  \eqref{vav01}.
\end{theorem}

The following weighted local estimates are valid.
\begin{theorem} \label{potv}
	Let $\Om \subset \Rn$ be an open unbounded  set, $0<\al<n$, $p\in \mathbb{P}^{log}_\infty(\Om)$, $p_+<\frac {n}{\al} $, $\frac1{q(x)}=\frac 1{p(x)}-\frac {\al} {n}$, $\om\in A_{p(\cdot),q(\cdot)}(\Om)$.
	Then
	\begin{equation}\label{newvK}
	\|I^{\al} f\|_{L^{q(\cdot)}_{\om}(\widetilde{B}(x,t))} \le
	C  \|\om\|_{L^{q(\cdot)}(\widetilde{B}(x,t))} \int_{t}^\infty  \|f\|_{L^{p(\cdot)}_{\om}(\widetilde{B}(x,s))}
	\|\om\|_{L^{q(\cdot)}(\widetilde{B}(x,s))}^{-1} \frac{ds}{s},
	\end{equation}
	where $C$ does not depend on $f$, $x$ and $t.$
\end{theorem}
\begin{proof}
	We represent  $f$ as
	\begin{equation}\label{repr}
	f=f_1+f_2, \ \quad f_1(y)=f(y)\chi _{\widetilde{B}(x,2t)}(y),\quad
	f_2(y)=f(y)\chi _{\Om\backslash \widetilde{B}(x,2t)}(y), \ \quad t>0,
	\end{equation}
	and have
	$$
	I^{\al} f(x)=I^{\al} f_1(x)+I^{\al} f_2(x).
	$$
	By  Corollary \ref{pvc} we obtain
	\begin{align*}
	\|I^{\al} f_1\|_{L^{q(\cdot)}_{\om}(\widetilde{B}(x,t))}&\le
	\|I^{\al} f_1\|_{L^{q(\cdot)}_{\om}(\Om)}
	\le C \|f_1\|_{L^{p(\cdot)}_{\om}(\Om)}=C \|f\|_{L^{p(\cdot)}_{\om}(\widetilde{B}(x,2t))}.
	\end{align*}
	Then
	\begin{gather*}
	\|I^{\al} f_1\|_{L^{q(\cdot)}_{\om}(\widetilde{B}(x,t))}
	\le C\|f\|_{L^{p(\cdot)}_{\om}(\widetilde{B}(x,2t))},
	\end{gather*}
	where the constant $C$ is independent of $f$.
	
	On the other hand,
	\begin{align}  \label{sal01}
	\|f\|_{L^{p(\cdot)}_{\om}(\widetilde{B}(x,2t))} & \thickapprox |B|^{1-\frac{\a}{n}} \|f\|_{L^{p(\cdot)}_{\om}(\widetilde{B}(x,2t))}
	\int_{2t}^{\i}\frac{ds}{s^{n+1-\a}} \notag
	\\
	& \le |B|^{1-\frac{\a}{n}}
	\int_{2t}^{\i}\|f\|_{L^{p(\cdot)}_{\om}(\widetilde{B}(x,s))}\frac{ds}{s^{n+1-\a}}
	\\
	& \lesssim \|\om\|_{L^{q(\cdot)}(\widetilde{B}(x,t))} \|w^{-1}\|_{L^{p'(\cdot)}(\widetilde{B}(x,t))} \,
	\int_{t}^{\i}\|f\|_{L^{p(\cdot)}_{\om}(\widetilde{B}(x,s))} \frac{ds}{s^{n+1-\a}}    \notag
	\\
	& \lesssim \|\om\|_{L^{q(\cdot)}(\widetilde{B}(x,t))}
	\int_{t}^{\i}\|f\|_{L^{p(\cdot)}_{\om}(\widetilde{B}(x,s))} \, \|\om^{-1}\|_{L^{p'(\cdot)}(\widetilde{B}(x,s))} \, \frac{ds}{s^{n+1-\a}} \notag
	\\
	& \lesssim [\om]_{A_{p(\cdot),q(\cdot)}} \,
	\|\om\|_{L^{q(\cdot)}(\widetilde{B}(x,t))} \,
	\int_{t}^{\i}\|f\|_{L^{p(\cdot)}_{\om}(\widetilde{B}(x,s))} \, \|\om\|_{L^{q(\cdot)}(\widetilde{B}(x,s))}^{-1}\, \frac{ds}{s}. \notag
	\end{align}
	
	Taking into account that
	$$
	\|f\|_{L^{p(\cdot)}_{\om}(\widetilde{B}(x,t))} \le
	C  \|\om\|_{L^{q(\cdot)}(\widetilde{B}(x,t))} \int_{t}^\infty  \|f\|_{L^{p(\cdot)}_{\om}(\widetilde{B}(x,s))}
	\|\om\|_{L^{q(\cdot)}(\widetilde{B}(x,s))}^{-1} \frac{ds}{s},
	$$
	we get
	\begin{equation}\label{Ga8v}
	\|I^{\al} f_1\|_{L^{q(\cdot)}_{\om}(\widetilde{B}(x,t))} \le
	C   \|\om\|_{L^{q(\cdot)}(\widetilde{B}(x,t))} \int_{t}^\infty  \|f\|_{L^{p(\cdot)}_{\om}(\widetilde{B}(x,s))}
	\|\om\|_{L^{q(\cdot)}(\widetilde{B}(x,s))}^{-1} \frac{ds}{s}.
	\end{equation}
	
	When $|x-z|\le t$, $|z-y|\ge 2t,$ we have $\frac{1}{2} |z-y| \le
	|x-y|\le\frac{3}{2} |z-y|$, and therefore
	\begin{align*}
	|I^{\al} f_2(x)| &\le
	\int_{\Om\backslash \widetilde{B}(x,2t)}|z-y|^{\al -n}|f(y)| dy
	\\
	&\le 2^{n-\a} \int_{\Om\backslash \widetilde{B}(x,2t)}|x-y|^{\al-n}|f(y)| dy.
	\end{align*}
	
	We obtain
	\begin{align*}
	\int_{\Om\backslash \widetilde{B}(x,2t)}  |f(y)| dy
	& = \int_{\Om\backslash \widetilde{B}(x,2t)}|f(y)|
	\left(\int_{|x-y|}^\infty s^{\al-n -1}ds\right)dy
	\\
	& \lesssim \int_{2t}^\infty s^{\al-n -1}
	\left(\int_{\{y \in \Om : 2t\le|x-y|\le s \}}|f(y)|dy\right)ds
	\\
	& \lesssim \int_{t}^\infty s^{\al-n -1}
	\|f\|_{L^{p(\cdot)}_{\om}(\widetilde{B}(x,s))}
	\|\om^{-1}\|_{L^{p'(\cdot)}(\widetilde{B}(x,s))}ds
	\\
	& \lesssim \int_{t}^\infty  \|f\|_{L^{p(\cdot)}_{\om}(\widetilde{B}(x,s))}
	\|\om\|_{L^{q(\cdot)}(\widetilde{B}(x,s))}^{-1} \frac{ds}{s}.
	\end{align*}
	
	Hence
	\begin{align*}
	\|I^{\al} f_2\|_{L^{q(\cdot)}_{\om}(\widetilde{B}(x,t))} & \lesssim  \|\om\|_{L^{q(\cdot)}(\widetilde{B}(x,t))} \int_{t}^\infty  \|f\|_{L^{p(\cdot)}_{\om}(\widetilde{B}(x,s))}
	\|\om\|_{L^{q(\cdot)}(\widetilde{B}(x,s))}^{-1} \frac{ds}{s},
	\end{align*}
	which together with \eqref{Ga8v} yields \eqref{newvK}.
\end{proof}

\begin{theorem}\label{M1Xv}
	Let $\Om \subset \Rn$ be an open unbounded  set, $0<\al<n$, $p\in \mathbb{P}^{log}_\infty(\Om)$, $p_+<\frac {n}{\al} $, $\frac1{q(x)}=\frac 1{p(x)}-\frac {\al} {n}$, $\om\in A_{p(\cdot),q(\cdot)}(\Om)$ and the functions $\vi_1(x,t)$ and $\vi_2(x,t)$ fulfill  the
	condition
	\begin{equation}\label{H1v}
	\int_{t}^{\infty}\frac{\ei_{s<r<\infty} \vi_1(x,r) \|\om\|_{L^{p(\cdot)}(\widetilde{B}(x,r))}} {\|\om\|_{L^{q(\cdot)}(\widetilde{B}(x,s))}}\frac{ds}{s}
	\lesssim \vi_2(x,t).
	\end{equation}
	Then the operator $I^{\al}$ is bounded from ${\cal M}^{p(\cdot),\vi_1}_{\om}(\Om)$ to
	${\cal M}^{q(\cdot),\vi_2}_{\om}(\Om)$.
\end{theorem}
\begin{proof} Let $\om\in A_{p(\cdot),q(\cdot)}(\Om)$, by condition \eqref{H1v} and Theorems \ref{potv}, \ref{GulJMS} with
	\linebreak$v_2(r)=\vi_2(x,r)^{-1}$, $v_1(r)=\vi_1(x,r)^{-1} \|\om\|_{L^{p(\cdot)}(\widetilde{B}(x,r))}^{-1}$,
	$g(r)=\|f\|_{L^{p(\cdot)}_{\om}(\widetilde{B}(x,r))}$ and $w(r)=\|\om\|_{L^{p(\cdot)}(\widetilde{B}(x,r))}^{-1} r^{-1}$ we obtain
	\begin{align*}
	&\|I^{\al} f\|_{{\cal M}^{q(\cdot),\vi_2}_{\om}(\Om)}
	\\
	& \lesssim \sup_{x\in \Om, \;t>0} \frac{ 1}{\vi_2(x,t)} \int_{t}^{\infty}\|f\|_{L^{p(\cdot)}_{\om}(\widetilde{B}(x,s))}
	\|\om\|_{L^{q(\cdot)}(\widetilde{B}(x,s))}^{-1}\frac{ds}{s}
	\\
	& \lesssim \sup_{x\in \Om, \;t>0} \frac{1}{\vi_1(x,t)\|\om\|_{L^{p(\cdot)}(\widetilde{B}(x,t))}}\|f\|_{L^{p(\cdot)}_{\om}(\widetilde{B}(x,t))}
	= \|f\|_{\mathcal{M}^{p(\cdot),\vi_1}_{\om}(\Om)}.
	\end{align*}
\end{proof}

Now we consider the commutators Riesz potential defined by
$$
[b, I^{\alpha}]f(x)=\int_{\Rn} [b(x)-b(y)]f(y)|x-y|^{\alpha-n} dy.
$$

The commutator generated by $M$ and a suitable function $b$ is formally defined by
$$
[M,b]f = M(bf) - bM(f).
$$

Given a measurable function $b$ the maximal commutator is defined by
$$
M_b(f)(x) := \sup\limits_{r>0}|B(x,r)|^{-1}\int_{B(x,r)}|b(x) - b(y)||f(y)|dy
$$
for all $x\in \mathbb{R}^n$.

\begin{lemma} \label{SGH1} \cite{FR} Let $b\in BMO(\Rn)$, $1<s<\infty$.  Then
	\begin{equation*}
	M^\sharp([b,I^{\al}]f(x))\le
	C\|b\|_{BMO}\left[\left(M|I^{\al}f(x)|^{s}\right)^{\frac1{s}}+
	\left(M^{s\al}|f(x)|^{s}\right)^{\frac1{s}}\right],
	\end{equation*}
	where $C>0$ is independed of $f$
	and $x$.
\end{lemma}

\begin{lemma} \label{KS} \cite{CrWa} Let $\Om \subset \Rn$ be an open unbounded  set, $p\in \mathbb{P}^{log}_\infty(\Om)$ and $\om\in A_{p(\cdot)}(\Om)$. Then
	$$
	\|f\om\|_{L^{p(\cdot)}}\le C \|\om M^\sharp f\|_{L^{p(\cdot)}}
	$$
	with a constant $C > 0$ not depending on $f$.
\end{lemma}

\begin{theorem} \label{mus} \cite[Theorem 1.13]{AGKM} Let $b \in BMO(\Rn)$. Suppose that $X$ is a Banach space of measurable
	functions defined on $\Rn$. Assume that $M$ is bounded on $X$. Then the operator $M_b$
	is bounded on $X$, and the inequality
	$$
	\|M_bf\|_X \le  C\|b\|_* \|f\|_X
	$$
	holds with constant $C$ independent of $f$.
\end{theorem}

\begin{corollary} \label{Maxkomv} Let  $b\in BMO(\Om)$, $p\in \mathbb{P}^{log}_\infty(\Om)$ and $\om\in A_{p(\cdot)}(\Om)$,
	then the operator $M_b$ is {bounded on} $L^{p(\cdot)}_{\om} (\Rn)$.
\end{corollary}

\begin{theorem}\label{M1c}  \cite{GulHasBad}
	Let $\Om \subset \Rn$ be an open unbounded  set,  $p\in \mathbb{P}^{log}_\infty(\Om)$, $\om\in A_{p(\cdot)}(\Om)$, $b\in BMO(\Om)$ and the function $\vi_1(x,r)$ and $\vi_2(x,r)$ satisfy the condition
	\begin{equation}\label{qhs1shk}
	\sup_{t>r}\left(1+\ln\frac{t}{r}\right)\frac{\ei_{t<s<\infty} \vi_1(x, s)
		\|\om\|_{L^{p(\cdot)}(\widetilde{B}(x,s))}}{\|\om\|_{L^{p(\cdot)}(\widetilde{B}(x,t))}}\le C\vi_2(x,r),
	\end{equation}
	where $C$ does not depend on $x\in\Om$ and $t$.
	Then the operator $M_b$ is
	bounded from the space $\mathcal{M}^{p(\cdot),\vi_1}_{\om}(\Om)$ to the space
	$\mathcal{M}^{p(\cdot),\vi_2}_{\om}(\Om)$.
\end{theorem}

\begin{theorem}\label{RKomv} Let $\Om \subset \Rn$ be an open unbounded  set, $0<\al<n$ and $p\in \mathbb{P}^{log}_\infty(\Om)$, $p_+<\frac {n}{\al} $, $\frac1{q(x)}=\frac 1{p(x)}-\frac {\al} {n}$, $\om\in A_{p(\cdot),q(\cdot)}(\Om)$.
	The following assertions are equivalent:
	
	(i) The operator $[b,I^{\al}]$ is bounded from $L^{p(\cdot)}_{\om}(\Om)$
	to $L^{q(\cdot)}_{\om}(\Om)$.
	
	(ii)  $b\in BMO(\Om)$.
\end{theorem}
\begin{proof}
	$(ii)\Rightarrow(i)$ Let $f\in L^{p(\cdot)}_{\om}(\Om)$ and $b\in BMO(\Om)$. By the Lemma \ref{KS}, we have
	\begin{gather*}
	\|[b,I^{\al}]f\|_{L^{q(\cdot)}_{\om}(\Om)}
	\lesssim \|M^\sharp([b,I^{\al}]f)\|_{L^{q(\cdot)}_{\om}(\Om)}.
	\end{gather*}

	From Lemma \ref{SGH1}, we have
	\begin{gather*}
	\|M^\sharp([b,I^{\al}]f)\|_{L^{q(\cdot)}_{\om}(\Om)} \lesssim \|b\|_{*}\left\|\left(M|I^{\al}f|^s
	\right)^{\frac{1}{s}}+ \left(M^{\al s}|f|^s
	\right)^{\frac{1}{s}}\right\|_{L^{q(\cdot)}_{\om}(\Om)}
	\\
	\lesssim \|b\|_{*} \left[\left\|\left(M|I^{\al} f|^s
	\right)^{\frac{1}{s}}\right\|_{L^{q(\cdot)}_{\om}(\Om)}+ \left\|\left(M^{\al s}|f|^s
	\right)^{\frac{1}{s}}\right\|_{L^{q(\cdot)}_{\om}(\Om)}\right].
	\end{gather*}
	
	By Theorem \ref{max2} and Corollary \ref{pvc}, we have
	\begin{align*}
	\left\|\left(M|I^{\al}f|^{s}\right)^{\frac1{s}}\right\|_{L^{q(\cdot)}_{\om}(\Om)}
	\lesssim \left\||I^{\al}f|^{s} \right\|_{L^{\frac{q(\cdot)}{s}}_{\om^{sq(\cdot)}}(\Om)}^{\frac1{s}}
	=  \left\|I^{\al}f\right\|_{L^{q(\cdot)}_{\om}(\Om)}
	\lesssim  \left\|f\right\|_{L^{p(\cdot)}_{\om}(\Om)}.
	\end{align*}
	
	By Corollary \ref{pvc}, we have
	\begin{gather*}
	\left\|\left(M^{\al s}|f|^s\right)^{\frac{1}{s}}\right\|_{L^{q(\cdot)}_{\om}(\Om)}
	\lesssim \left\|f\right\|_{L^{p(\cdot)}_{\om}(\Om)}.
	\end{gather*}
	
	Therefore
	\begin{gather*}
	\|[b,I^{\al}]f\|_{L^{q(\cdot)}_{\om}(\Om)}
	\lesssim \|b\|_{*}\left\|f\right\|_{L^{p(\cdot)}_{\om}(\Om)}.
	\end{gather*}

	
	$(i)\Rightarrow(ii)$  Now, let us prove the "only if" part. Let $[b,I^{\alpha}]$ be bounded from
	$L^{p(\cdot)}_{\om}(\Om)$ to $L^{q(\cdot)}_{\om}(\Om)$, $1< p_+<\frac
	{n}{\alpha}$. Then
	\begin{align*}
	&|B(x,t)| \int_{\widetilde{B}(x,t)}|b(z)-b_{B(x,t)}|dz
	\\
	&=\frac1{|B(x,t)|} \int\limits_{\widetilde{B}(x,t)}\Big|b(z)-\frac1{|B(x,t)|}
	\int\limits_{\widetilde{B}(x,t)}b(y) dy \Big|dz\\
	&\leq\frac1{|B(x,t)|^{1+\frac{\alpha}{n}}} \int\limits_{\widetilde{B}(x,t)}\frac1{|B(x,t)|^{1-\frac{\alpha}{n}}}
	\Big|\int\limits_{\widetilde{B}(x,t)}\left(b(z)-b(y)\right)dy\Big| dz\\
	&\leq\frac1{|B(x,t)|^{1+\frac{\alpha}{n}}} \int\limits_{\widetilde{B}(x,t)}
	\Big|\int\limits_{\widetilde{B}(x,t)}\left(b(z)-b(y)\right)|x-y|^{\alpha-n} dy \Big|dz\\
	&\leq\frac1{|B(x,t)|^{1+\frac{\alpha}{n}}} \int\limits_{\widetilde{B}(x,t)}\left|[b,I_{\alpha}]\chi_{B(x,t)}(z) \right| dz
	\\
	&\leq C t^{-n-\alpha}\|[b,I_{\alpha}]\chi_{B(x,t)}\|_{L^{q(\cdot)}_{\om}}
	\|\chi_{B(x,t)}\|_{L^{q'(\cdot)}_{\om^{-1}}}
	\\
	&\leq C t^{-n-\alpha}\|\om\|_{L^{p(\cdot)}(B(x,t))}
	\|\om^{-1}\|_{L^{q'(\cdot)}(B(x,t))}  \leq C.
	\end{align*}
	
	Hence we get
	\begin{gather*}
	|B(x,t)|^{-1} \int_{\widetilde{B}(x,t)}|b(y)-b_{B(x,t)}| dy  \le C.
	\end{gather*}
	
	This shows that $b\in BMO(\Om)$.

	The theorem has been proved.
\end{proof}

\begin{theorem} \label{potvk}
	Let $\Om \subset \Rn$ be an open unbounded  set, $0<\al<n$, $p\in \mathbb{P}^{log}_\infty(\Om)$,  $p_+<\frac {n}{\al} $, $\frac1{q(x)}=\frac 1{p(x)}-\frac {\al} {n}$, $\om\in A_{p(\cdot),q(\cdot)}(\Om)$, $b\in BMO(\Om)$.
	Then
	$$
	\|[b,I^{\alpha}] f\|_{L^{q(\cdot)}_{\om}(\widetilde{B}(x,t))}\le
	C\|b\|_{*}\|\om\|_{L^{q(\cdot)}(\widetilde{B}(x,t))}
	$$
	\begin{equation}\label{newvk}
	\times
	\int_{t}^{\infty}\left(1+\ln\frac{s}{t}\right)
	\|f\|_{L^{p(\cdot)}_{\om}(\widetilde{B}(x,s))}
	\|\om\|_{L^{q(\cdot)}(\widetilde{B}(x,s))}^{-1}\frac{ds}{s},
	\end{equation}
	where $C$ does not depend on $f$, $x$ and $t.$
\end{theorem}
\begin{proof}
	We represent  $f$ as
	\begin{equation}\label{reprx}
	f=f_1+f_2, \ \quad f_1(y)=f(y)\chi _{\widetilde{B}(x,2t)}(y),\quad
	f_2(y)=f(y)\chi _{\Om\backslash \widetilde{B}(x,2t)}(y), \ \quad t>0,
	\end{equation}
	and have
	$$
	[b,I^{\alpha}] f(x)=[b,I^{\alpha}] f_1(x)+[b,I^{\alpha}] f_2(x).
	$$
	By  Theorem \ref{RKomv} we obtain
	\begin{align*}
	&\|[b,I^{\alpha}] f_1\|_{L^{q(\cdot)}_{\om}(\widetilde{B}(x,t))}\le
	\|[b,I^{\alpha}] f_1\|_{L^{q(\cdot)}_{\om}(\Om)}
	\\
	& \lesssim \|b\|_{*}\|f_1\|_{L^{p(\cdot)}_{\om}(\Om)} = \|b\|_{*}\|f\|_{L^{p(\cdot)}_{\om}(\widetilde{B}(x,2t))}.
	\end{align*}
	Then
	\begin{gather*}
	\|[b,I^{\alpha}] f_1\|_{L^{q(\cdot)}_{\om}(\widetilde{B}(x,t))}
	\le C\|b\|_{*}\|f\|_{L^{p(\cdot)}_{\om}(\widetilde{B}(x,2t))},
	\end{gather*}
	where the constant $C$ is independent of $f$.
	
	Taking into account that from the inequality \eqref{sal01} we have
	$$
	\|f\|_{L^{p(\cdot)}_{\om}(\widetilde{B}(x,t))} \le
	C  \|b\|_{*}\|\om\|_{L^{q(\cdot)}(\widetilde{B}(x,t))} \int_{t}^\infty  \|f\|_{L^{p(\cdot)}_{\om}(\widetilde{B}(x,s))}
	\|\om\|_{L^{q(\cdot)}(\widetilde{B}(x,s))}^{-1}\frac{ds}{s},
	$$
	and then
	\begin{equation}\label{Ga8vk}
	\|[b,I^{\alpha}] f_1\|_{L^{q(\cdot)}_{\om}(\widetilde{B}(x,t))} \le
	C  \|b\|_{*}\|\om\|_{L^{q(\cdot)}(\widetilde{B}(x,t))} \int_{t}^\infty  \|f\|_{L^{p(\cdot)}_{\om}(\widetilde{B}(x,s))}
	\|\om\|_{L^{q(\cdot)}(\widetilde{B}(x,s))}^{-1}\frac{ds}{s}.
	\end{equation}
	
	When $|x-z|\le t$, $|z-y|\ge 2t,$ we have $\frac{1}{2} |z-y| \le
	|x-y|\le\frac{3}{2} |z-y|$, and therefore
	\begin{align*}
	|[b,I^{\alpha}] f_2(x)| &\le
	\int_{\Om\backslash \widetilde{B}(x,2t)}|b(y)-b(z)||z-y|^{\al -n}|f(y)| dy
	\\
	&\le C \int_{\Om\backslash \widetilde{B}(x,2t)}|b(y)-b(z)||x-y|^{\al-n}|f(y)| dy.
	\end{align*}
	
	We  obtain
	\begin{align*}
	&\int_{\Om\backslash \widetilde{B}(x,2t)}|b(y)-b(z)||x-y|^{\al-n}  |f(y)| dy
	\\
	&= \int_{\Om\backslash \widetilde{B}(x,2t)}|b(y)-b(z)||f(y)|
	\left(\int_{|x-y|}^\infty s^{\al-n -1}ds\right)dy
	\\
	&\le C \int_{2t}^\infty s^{\al-n -1}
	\left(\int_{\{y \in \Om : 2t\le|x-y|\le s \}}|b(y)-b_{\widetilde{B}(x,t)}||f(y)|dy\right)ds
	\\
	&+ C |b(z)-b_{\widetilde{B}(x,t)}|\int_{2t}^\infty s^{\al-n -1}
	\left(\int_{\{y \in \Om : 2t\le|x-y|\le s \}}|f(y)|dy\right)ds=V_1+V_2.
	\end{align*}
	
	To estimate $V_1$:
	\begin{align}\label{v-1}
	&V_1=C\int_{2t}^\infty s^{\al-n -1}
	\left(\int_{\{y \in \Om : 2t\le|x-y|\le s \}}|b(y)-b_{\widetilde{B}(x,t)}||f(y)|dy\right)ds\notag
	\\
	& \le C\int_{t}^\infty s^{\al-n -1}\|b(\cdot) - b_{\widetilde{B}(x,s)}\|_{L^{p'(\cdot)}_{\om^{-1}}(\widetilde{B}(x,s))}
	\|f\|_{L^{p(\cdot)}_{\om}(\widetilde{B}(x,s))}ds\notag
	\\
	& + C\int_{t}^\infty s^{\al-n -1}|b_{\widetilde{B}(x,t)}-b_{\widetilde{B}(x,s)}|
	\left(\int_{ \widetilde{B}(x,s)}|f(y)|dy\right)ds\notag
	\\
	&\le C\|b\|_{*}\int_{t}^{\infty} s^{\al-n -1}
	\|\om^{-1}\|_{L^{p'(\cdot)}(\widetilde{B}(x,s))}
	\|f\|_{L^{p(\cdot)}_{\om}(\widetilde{B}(x,s))}ds \notag
	\\
	&+ C\|b\|_{*}\int_{t}^{\infty} s^{\al-n -1}\ln\frac{s}{t}
	\|\om^{-1}\|_{L^{p'(\cdot)}(\widetilde{B}(x,s))}
	\|f\|_{L^{p(\cdot)}_{\om}(\widetilde{B}(x,s))}ds \notag
	\\
	&\le  C\|b\|_{*}\int_{t}^{\infty} \left(1+\ln\frac{s}{t}\right)
	\|\om\|_{L^{p(\cdot)}(\widetilde{B}(x,s))}^{-1}
	\|f\|_{L^{p(\cdot)}_{\om}(\widetilde{B}(x,s))}\frac{ds}{s} .
	\end{align}
	
	To estimate $V_2$:
	\begin{align}\label{v-2}
	V_2=&C |b(z)-b_{\widetilde{B}(x,t)}|\int_{2t}^\infty s^{\al-n -1}
	\left(\int_{\{y \in \Om : 2t\le|x-y|\le s \}}|f(y)|dy\right)ds \notag
	\\
	&\le C|B(x,t)|^{-1}\int_{\widetilde{B}(x,t)}|b(z) - b(y)|dy\int_{2t}^{\infty}
	\|f\|_{L^{p(\cdot)}_{\om}(\widetilde{B}(x,s))}
	\|\om\|_{L^{q(\cdot)}(\widetilde{B}(x,s))}^{-1}\frac{ds}{s}
	\notag
	\\
	&\le CM_b\chi_{B(x,t)}(z)\int_{t}^{\infty} \left(1+\ln\frac{s}{t}\right)
	\|f\|_{L^{p(\cdot)}_{\om}(\widetilde{B}(x,s))}
	\|\om\|_{L^{q(\cdot)}(\widetilde{B}(x,s))}^{-1}\frac{ds}{s},
	\end{align}
	where $C$ does not depend on $x,t$.  Then by Corollary \ref{Maxkomv} and \eqref{v-1}, \eqref{v-2} we  have
	$$
	\|[b,I^{\alpha}]f_2\|_{L^{q(\cdot)}_{\om}(\widetilde{B}(x,t))} \le \|V_1\|_{L^{q(\cdot)}_{\om}(\widetilde{B}(x,t))}+\|V_2\|_{L^{q(\cdot)}_{\om}(\widetilde{B}(x,t))}
	$$
	$$
	\le C\|b\|_{*}\|\om\|_{L^{q(\cdot)}(\widetilde{B}(x,t))}\int_{t}^{\infty}\left(1+\ln\frac{s}{t}\right)
	\|f\|_{L^{p(\cdot)}_{\om}(\widetilde{B}(x,s))}
	\|\om\|_{L^{q(\cdot)}(\widetilde{B}(x,s))}^{-1}\frac{ds}{s}
	$$
	$$
	+C\|M_b\chi_{B(x,t)}\|_{L^{q(\cdot)}_{\om}(\widetilde{B}(x,t))}
	\int_{t}^{\infty}\left(1+\ln\frac{s}{t}\right)
	\|f\|_{L^{p(\cdot)}_{\om}(\widetilde{B}(x,s))}
	\|\om\|_{L^{q(\cdot)}(\widetilde{B}(x,s))}^{-1}\frac{ds}{s}
	$$
	$$
	\le C\|b\|_{*}\|\om\|_{L^{q(\cdot)}(\widetilde{B}(x,t))}\int_{t}^{\infty}\left(1+\ln\frac{s}{t}\right)
	\|f\|_{L^{p(\cdot)}_{\om}(\widetilde{B}(x,s))}
	\|\om\|_{L^{q(\cdot)}(\widetilde{B}(x,s))}^{-1}\frac{ds}{s}
	$$
	$$
	+C\|b\|_{*}\|\om\|_{L^{q(\cdot)}(\widetilde{B}(x,t))}
	\int_{t}^{\infty}\left(1+\ln\frac{s}{t}\right)
	\|f\|_{L^{p(\cdot)}_{\om}(\widetilde{B}(x,s))}
	\|\om\|_{L^{q(\cdot)}(\widetilde{B}(x,s))}^{-1}\frac{ds}{s}
	$$
	$$
	\le C\|b\|_{*}\|\om\|_{L^{q(\cdot)}(\widetilde{B}(x,t))}
	\int_{t}^{\infty}\left(1+\ln\frac{s}{t}\right)
	\|f\|_{L^{p(\cdot)}_{\om}(\widetilde{B}(x,s))}
	\|\om\|_{L^{q(\cdot)}(\widetilde{B}(x,s))}^{-1}\frac{ds}{s}.
	$$
	
	Hence
	\begin{align*}
	&\|[b,I^{\alpha}] f_2\|_{L^{q(\cdot)}_{\om}(\widetilde{B}(x,t))}
	\\
	\le & C\|b\|_{*}\|\om\|_{L^{q(\cdot)}(\widetilde{B}(x,t))}
	\int_{t}^{\infty}\left(1+\ln\frac{s}{t}\right)
	\|f\|_{L^{p(\cdot)}_{\om}(\widetilde{B}(x,s))}
	\|\om\|_{L^{q(\cdot)}(\widetilde{B}(x,s))}^{-1}\frac{ds}{s},
	\end{align*}
	which together with \eqref{Ga8vk} yields \eqref{newvk}.
\end{proof}

\begin{theorem}\label{M1Xvc}
	Let $\Om \subset \Rn$ be an open unbounded  set, $0<\al<n$, $p\in \mathbb{P}^{log}_\infty(\Om)$, $p_+<\frac {n}{\al} $,
	$\frac1{q(x)}=\frac 1{p(x)}-\frac {\al} {n}$, $\om\in A_{p(\cdot),q(\cdot)}(\Om)$, $b\in BMO(\Om)$ and the functions $\vi_1(x,t)$ and $\vi_2(x,t)$ fulfill  the
	condition
	\begin{equation}\label{H1vk}
	\int_{t}^{\infty}\left(1+\ln\frac{s}{t}\right)\frac{\ei_{s<r<\i} \vi_1(x, r) \|\om\|_{L^{p(\cdot)}(\widetilde{B}(x,r))}} {\|\om\|_{L^{q(\cdot)}(\widetilde{B}(x,s))}}\frac{ds}{s}
	\le C \vi_2(x,t).
	\end{equation}
	Then the operators $[b,I^{\alpha}]$ is bounded from ${\cal M}^{p(\cdot),\vi_1}_{\om}(\Om)$ to
	${\cal M}^{q(\cdot),\vi_2}_{\om}(\Om)$.
\end{theorem}
\begin{proof} Let $\om\in A_{p(\cdot),q(\cdot)}(\Om)$, by condition \eqref{H1vk} and Theorems \ref{potvk}, \ref{thm3.2.} with
	\linebreak	$v_2(r)=\vi_2(x,r)^{-1}$, $v_1(r)=\vi_1(x,r)^{-1} \|\om\|_{L^{p(\cdot)}(\widetilde{B}(x,r))}^{-1}$,
	$g(r)=\|f\|_{L^{p(\cdot)}_{\om}(\widetilde{B}(x,r))}$ and $w(r)=\|\om\|_{L^{p(\cdot)}(\widetilde{B}(x,r))}^{-1} r^{-1}$ we obtain
	\begin{align*}
	&\|[b,I^{\alpha}] f\|_{{\cal M}^{q(\cdot),\vi_2}_{\om}(\Om)}
	\\
	& \lesssim \|b\|_{*}\|\sup_{x\in \Om, \, t>0} \frac{1}{\vi_2(x,t)} \int_{t}^\infty \left(1+\ln\frac{s}{t}\right) \|f\|_{L^{p(\cdot)}_{\om}(\widetilde{B}(x,s))}
	\|\om\|_{L^{q(\cdot)}(\widetilde{B}(x,s))}^{-1}\frac{ds}{s}
	\\
	& \lesssim \|b\|_{*}\sup_{x\in \Om, \, t>0}
	\frac{1}{\vi_1(x,t)\|\om\|_{L^{p(\cdot)}(\widetilde{B}(x,t))}} \,
	\|f\|_{L^{p(\cdot)}_{\om}(\widetilde{B}(x,t))} = \|b\|_{*}\|\|f\|_{{\cal M}^{p(\cdot),\vi_1}_{\om}(\Om)}.
	\end{align*}
\end{proof}

\

\section{Riesz potential  and its commutators in the spaces
	$V\mathcal{M}^{p(\cdot),\vi}_{\om}(\Om)$}\label{potentialsv}
\setcounter{theorem}{0} \setcounter{equation}{0}

\begin{theorem}\label{Risv}
	Let $\Om \subset \Rn$ be an open unbounded  set, $0<\al<n$, $p\in \mathbb{P}^{log}_\infty(\Om)$, $p_+<\frac {n}{\al} $, $\frac1{q(x)}=\frac 1{p(x)}-\frac {\al} {n}$, $\om\in A_{p(\cdot),q(\cdot)}(\Om)$ and the functions $\vi_1(x,t)$ and $\vi_2(x,t)$ fulfill  the
	conditions
	\begin{equation}\label{rv'}
	C_{\gamma_0}:=\int_{\gamma_0}^{\i} \frac{\ei_{s<r<\infty} \vi_1(x,r) \|\om\|_{L^{q(\cdot)}(\widetilde{B}(x,r))}} {\|\om\|_{L^{q(\cdot)}(\widetilde{B}(x,s))}} \frac{ds}{s} <\infty
	\end{equation}
	for every $\gamma_0> 0$, and
	\begin{equation}\label{rv}
	\int_{t}^{\infty} \frac{\ei_{s<r<\infty} \vi_1(x,r) \|\om\|_{L^{q(\cdot)}(\widetilde{B}(x,r))}} {\|\om\|_{L^{q(\cdot)}(\widetilde{B}(x,s))}} \frac{ds}{s}
	\le C \vi_2(x,t).
	\end{equation}
	Then the operators $I^{\al}$ is bounded from $V{\cal M}^{p(\cdot),\vi_1}_{\om}(\Om)$ to
	$V{\cal M}^{q(\cdot),\vi_2}_{\om}(\Om)$.
\end{theorem}
\begin{proof} The norm inequalities follow from Theorem \ref{M1Xv}, so we only have to prove that if
	$$
	\lim\limits_{r\to 0}\sup\limits_{x\in\Rn} \frac{1}{\vi_1(x,t)\|\om\|_{L^{q(\cdot)}(\widetilde{B}(x,t))}} \|f\chi_{\widetilde{B}(x,t)}\|_{L^{p(\cdot)}_{\om}(\Om)}=0,
	$$
	then
	\begin{equation}\label{rv2}
	\lim\limits_{r\to 0}\sup\limits_{x\in\Rn} \frac{1}{\vi_2(x,t)\|\om\|_{L^{q(\cdot)}(\widetilde{B}(x,t))}} \| I^{\al}
	f\chi_{\widetilde{B}(x,t)}\|_{L^{q(\cdot)}_{\om}(\Om)}=0
	\end{equation}
	otherwise. 
	
	To show that $\sup\limits_{x\in\Rn} \frac{1}{\vi_2(x,t)\|\om\|_{L^{q(\cdot)}(\widetilde{B}(x,t))}} \| I^{\al}
	f\chi_{\widetilde{B}(x,t)}\|_{L^{q(\cdot)}_{\om}(\Om)}<\varepsilon$ for small $r$, we split the right-hand
	side of \eqref{newvK}:
	\begin{equation}\label{rv3}
	\sup\limits_{x\in\Rn} \frac{1}{\vi_2(x,t)\|\om\|_{L^{q(\cdot)}(\widetilde{B}(x,t))}} \| I^{\al}
	f\chi_{\widetilde{B}(x,t)}\|_{L^{p(\cdot)}_{\om}(\Om)}\le C_0\left(I_{1,\gamma_0}(x,t)+I_{2,\gamma_0}(x,t)\right),
	\end{equation}
	where $\gamma_0 > 0$ will be chosen as shown below (we may take $\gamma_0 < 1$),
	$$
	I_{1,\gamma_0}(x,t) := \|\om\|_{L^{q(\cdot)}(\widetilde{B}(x,t))} \int_{t}^{\gamma_0}  \|f\|_{L^{p(\cdot)}_{\om}(\widetilde{B}(x,s))}
	\|\om\|_{L^{q(\cdot)}(\widetilde{B}(x,s))}^{-1} \frac{ds}{s},
	$$
	$$
	I_{2,\gamma_0}(x,t) := \|\om\|_{L^{q(\cdot)}(\widetilde{B}(x,t))} \int_{\gamma_0}^\infty  \|f\|_{L^{p(\cdot)}_{\om}(\widetilde{B}(x,s))}
	\|\om\|_{L^{q(\cdot)}(\widetilde{B}(x,s))}^{-1} \frac{ds}{s},
	$$
	and it is supposed that $t < \gamma_0$. Now we choose any fixed $\gamma_0 > 0$ such that
	$$
	\sup\limits_{x\in\Rn} \frac{1}{\vi_1(x,t)\|\om\|_{L^{q(\cdot)}(\widetilde{B}(x,t))}} \|f\chi_{\widetilde{B}(x,t)}\|_{L^{p(\cdot)}_{\om}(\Om)}<
	\frac{\varepsilon}{2CC_0},\,\, \mbox{for all} \,\,0 < t < \gamma_0,
	$$
	where $C$ and $C_0$ are constants from \eqref{rv} and \eqref{rv3}, which is possible since $f\in V{\cal M}^{p(\cdot),\vi_1}_{\om}(\Om)$. Then
	$$
	\sup\limits_{x\in\Rn} CI_{1,\gamma_0}(x,t)<\frac{\varepsilon}{2}, \,\,0 < t < \gamma_0,
	$$
	by \eqref{rv}.
	
	The estimation of the second term now may be made already by the choice of $t$
	sufficiently small thanks to the condition \eqref{rv'}. We have
	$$
	I_{2,\gamma_0}(x,t)\le C_{\gamma_0}\frac{\vi_2(x,t)}{\|\om\|_{L^{q(\cdot)}(\widetilde{B}(x,t))}}\|f\|_{V{\cal M}^{q(\cdot),\vi_2}_{\om}(\Om)},
	$$
	where $C_{\gamma_0}$ is the constant from \eqref{rv'}. Then, by \eqref{rv'} it suffices to choose $r$ small enough such that
	$$
	\frac{\vi_2(x,t)}{\|\om\|_{L^{q(\cdot)}(\widetilde{B}(x,t))}}<\frac{\varepsilon}{2CC_{\gamma_0}\|f\|_{V{\cal M}^{q(\cdot),\vi_2}_{\om}(\Om)}},
	$$
	which completes the proof of \eqref{rv2}.
\end{proof}

\begin{theorem}\label{CRisv}
	Let $\Om \subset \Rn$ be an open unbounded  set, $0<\al<n$, $p\in \mathbb{P}^{log}_\infty(\Om)$, $p_+<\frac {n}{\al} $, $\frac1{q(x)}=\frac 1{p(x)}-\frac {\al} {n}$, $\om\in A_{p(\cdot),q(\cdot)}(\Om)$ and the functions $\vi_1(x,t)$ and $\vi_2(x,t)$ fulfill  the
	conditions
	\begin{equation}\label{rv4}
	C_{\gamma}:=\int_{\gamma}^{\infty}\left(1+\ln\frac{s}{t}\right)\frac{\ei_{s<r<\infty} \vi_1(x,r) \|\om\|_{L^{q(\cdot)}(\widetilde{B}(x,r))}} {\|\om\|_{L^{q(\cdot)}(\widetilde{B}(x,s))}} \frac{ds}{s}
	<\infty
	\end{equation}
	for every $\gamma$, and
	\begin{equation}\label{rv5}
	\int_{t}^{\infty}\left(1+\ln\frac{s}{t}\right)\frac{\ei_{s<r<\infty} \vi_1(x,r) \|\om\|_{L^{q(\cdot)}(\widetilde{B}(x,r))}} {\|\om\|_{L^{q(\cdot)}(\widetilde{B}(x,s))}} \frac{ds}{s}
	\le C \vi_2(x,t).
	\end{equation}
	Then the operators $[b,I^{\alpha}]$ is bounded from $V{\cal M}^{p(\cdot),\vi_1}_{\om}(\Om)$ to
	$V{\cal M}^{q(\cdot),\vi_2}_{\om}(\Om)$.
\end{theorem}

\begin{proof} The norm inequalities follow from Theorem \ref{M1Xvc}, so we only have to prove that
	$$
	\lim\limits_{r\to 0}\sup\limits_{x\in\Rn} \frac{1}{\vi_1(x,t)\|\om\|_{L^{q(\cdot)}(\widetilde{B}(x,t))}} \|f\chi_{\widetilde{B}(x,t)}\|_{L^{p(\cdot)}_{\om}(\Om)}=0\,\, \Rightarrow
	$$
	\begin{equation}\label{rv6}
	\lim\limits_{r\to 0}\sup\limits_{x\in\Rn} \frac{1}{\vi_2(x,t)\|\om\|_{L^{q(\cdot)}(\widetilde{B}(x,t))}} \|[b,I^{\alpha}]f\chi_{\widetilde{B}(x,t)}\|_{L^{q(\cdot)}_{\om}(\Om)}=0
	\end{equation}
	otherwise.
	
	To show that $\sup\limits_{x\in\Rn} \frac{1}{\vi_2(x,t)\|\om\|_{L^{q(\cdot)}(\widetilde{B}(x,t))}} \|[b,I^{\alpha}]
	f\chi_{\widetilde{B}(x,t)}\|_{L^{q(\cdot)}_{\om}(\Om)}<\varepsilon$ for small $r$, we split the right-hand
	side of \eqref{newvk}:
	\begin{equation}\label{rv7}
	\sup\limits_{x\in\Rn} \frac{1}{\vi_2(x,t)\|\om\|_{L^{q(\cdot)}(\widetilde{B}(x,t))}} \|[b,I^{\alpha}]
	f\chi_{\widetilde{B}(x,t)}\|_{L^{q(\cdot)}_{\om}(\Om)}\le C_0\left(I_{1,\gamma}(x,r)+I_{2,\gamma}(x,r)\right),
	\end{equation}
	where $\gamma> 0$ will be chosen as shown below (we may take $\gamma < 1$),
	$$
	I_{1,\gamma}(x,t) := \|b\|_{*}\|\om\|_{L^{q(\cdot)}(\widetilde{B}(x,t))} \int_{t}^{\gamma}  \left(1+\ln\frac{s}{t}\right)
	\|f\|_{L^{p(\cdot)}_{\om}(\widetilde{B}(x,s))}
	\|\om\|_{L^{q(\cdot)}(\widetilde{B}(x,s))}^{-1}\frac{ds}{s},
	$$
	$$
	I_{2,\gamma}(x,t) := \|b\|_{*}\|\om\|_{L^{q(\cdot)}(\widetilde{B}(x,t))} \int_{\gamma}^\infty  \left(1+\ln\frac{s}{t}\right)
	\|f\|_{L^{p(\cdot)}_{\om}(\widetilde{B}(x,s))}
	\|\om\|_{L^{q(\cdot)}(\widetilde{B}(x,s))}^{-1}\frac{ds}{s}
	$$
	and it is supposed that $t < \gamma$. Now we choose any fixed $\gamma> 0$ such that
	$$
	\sup\limits_{x\in\Rn} \frac{1}{\vi_1(x,t)\|\om\|_{L^{q(\cdot)}(\widetilde{B}(x,t))}} \|f\chi_{\widetilde{B}(x,t)}\|_{L^{p(\cdot)}_{\om}(\Om)}<
	\frac{\varepsilon}{2CC_0\|b\|_{*}},\,\, \mbox{for all} \,\,0 < t < \gamma,
	$$
	where $C$ and $C_0$ are constants from \eqref{rv5} and \eqref{rv7}, which is possible since $f\in V{\cal M}^{p(\cdot),\vi_1}_{\om}(\Om)$. Then
	$$
	\sup\limits_{x\in\Rn} CI_{1,\gamma}(x,t)<\frac{\varepsilon}{2}, \,\,0 < t < \gamma,
	$$
	by \eqref{rv5}.
	
	The estimation of the second term now may be made already by the choice of $ r$
	sufficiently small thanks to the condition \eqref{rv4}. We have
	$$
	I_{2,\gamma}(x,t)\le C_{\gamma}\|b\|_{*}\frac{\vi_2(x,t)}{\|\om\|_{L^{q(\cdot)}(\widetilde{B}(x,t))}}
	\|f\|_{V{\cal M}^{q(\cdot),\vi_2}_{\om}(\Om)},
	$$
	where $C_{\gamma}$ is the constant from \eqref{rv4}. Then, by \eqref{rv4} it suffices to choose $r$ small enough such that
	$$
	\frac{\vi_2(x,t)}{\|\om\|_{L^{q(\cdot)}(\widetilde{B}(x,t))}} < \frac{\varepsilon}{2CC_{\gamma}\|b\|_{*}\|f\|_{V{\cal M}^{q(\cdot),\vi_2}_{\om}(\Om)}},
	$$
	which completes the proof of \eqref{rv5}.
\end{proof}

\

{\it Acknowledgment}.
We thank the referee(s) for careful reading the paper and useful comments.
The research of V.S. Guliyev was partially supported by the Ministry of Education and Science of the Russian Federation (the Agreement number: 02.a03.21.0008) and by the grant of 1st Azerbaijan-Russia Joint Grant Competition (Grant No. EIFBGM-4-RFTF-1/2017-21/01/1).

\

\end{document}